\documentclass[a4paper,12pt]{amsart}
\usepackage{hyperref,fancyhdr,mathrsfs,amsmath,amscd,amsthm,amsfonts,latexsym,amssymb}

\DeclareMathOperator{\dif}{d}

\newcommand{\Cal}{\mathcal{C}}
\renewcommand{\H}{\mathscr{H}}
\newcommand{\V}{\mathscr{V}}
\newcommand{\Q}{\mathscr{Q}}
\newcommand{\F}{\mathscr{F}}
\newcommand{\Fa}{\mathcal{F}}
\newcommand{\J}{\mathcal{J}}

\def \a{\alpha}

\def \b{\beta}

\def \phi{\varphi}
\def \Phi{\varPhi}
\def \p{\pi}
\def \r{\rho}
\def \s{\sigma}
\def \t{\tau}
\def \D{\Delta}
\def \R{\mathbb{R}}
\def \Hq{\mathbb{H}\,}
\def \C{\mathbb{C}\,}

\def\widecheckg{g^{\hspace*{-2.5pt}\vbox to 5pt{\hbox to
0pt{\LARGE$\check{}$}}}\hspace*{2pt}}

\def\widecheckl{\lambda^{\hspace*{-3.5pt}\vbox to 8pt{\hbox to
0pt{\LARGE$\check{}$}}}\hspace*{2pt}}

\begin{document}

\title{Twistorial maps between quaternionic manifolds}
\author{S.~Ianu\c s, S.~Marchiafava, L.~Ornea, R.~Pantilie}
\thanks{S.I., L.O.\ and R.P.\ acknowledge that this work was partially supported by a
CEx Grant no.\ 2-CEx 06-11-22/25.07.2006.\\
\indent
S.M. acknowledges that this work was done under the program of GNSAGA-INDAM of
C.N.R. and PRIN05 ''Geometria Riemanniana e strutture
differenziabili'' of MIUR (Italy).}
\email{\href{mailto:istere@yahoo.com}{istere@yahoo.com},
       \href{mailto:marchiaf@mat.uniroma1.it}{marchiaf@mat.uniroma1.it},
       \href{mailto:liviu.ornea@imar.ro}{liviu.ornea@imar.ro},
       \href{mailto:radu.pantilie@imar.ro}{radu.pantilie@imar.ro}}
\address{S.~Ianu\c s, Universitatea din Bucure\c sti, Facultatea de Matematic\u a, Str.\ Academiei nr.\ 14,
70109, Bucure\c sti, Rom\^ania}
\address{S.~Marchiafava, Dipartimento di Matematica, Istituto ``Guido~Castelnuovo'',
Universit\`a degli Studi di Roma ``La Sapienza'', Piazzale Aldo~Moro, 2 - I 00185 Roma - Italia}
\address{L.~Ornea, Universitatea din Bucure\c sti, Facultatea de Matematic\u a, Str.\ Academiei nr.\ 14,
70109, Bucure\c sti, Rom\^ania, \emph{also,}
 Institutul de Matematic\u a ``Simion~Stoilow'' al Academiei Rom\^ane,
C.P. 1-764, 014700, Bucure\c sti, Rom\^ania}
\address{R.~Pantilie, Institutul de Matematic\u a ``Simion~Stoilow'' al Academiei Rom\^ane,
C.P. 1-764, 014700, Bucure\c sti, Rom\^ania}

\newtheorem{thm}{Theorem}[section]
\newtheorem{lem}[thm]{Lemma}
\newtheorem{cor}[thm]{Corollary}
\newtheorem{prop}[thm]{Proposition}

\theoremstyle{definition}

\newtheorem{defn}[thm]{Definition}
\newtheorem{rem}[thm]{Remark}
\newtheorem{exm}[thm]{Example}

\numberwithin{equation}{section}

\maketitle
\thispagestyle{empty}
\vspace{-4mm}
\begin{center}
\emph{This paper is dedicated to the memory of Kris Galicki.}
\end{center}

\section*{Abstract}
\begin{quote}
{\footnotesize  We introduce a natural notion of \emph{quaternionic map} between almost quaternionic
manifolds and we prove the following, for maps of rank at least one:\\
\indent
$\bullet$ A map between quaternionic manifolds endowed with the integrable
almost twistorial structures is twistorial if and only if it is quaternionic.\\
\indent
$\bullet$ A map between quaternionic manifolds endowed with the nonintegrable almost twistorial structures
is twistorial if and only if it is quaternionic and totally-geodesic.\\
As an application, we describe all the quaternionic maps between open sets of quaternionic projective spaces.}
\end{quote}

\section*{Introduction}

\indent
An \emph{almost quaternionic structure} on a manifold is a reduction of its frame bundle to
the group ${\rm Sp}(1)\cdot{\rm GL}(m,\Hq)$\,. The integrability condition for an almost quaternionic
structure (that is, the condition that the corresponding reduction of the frame bundle be given
by the cocycle determined by an atlas) is very restrictive \cite{Mar-70}\,. Nevertheless, as
${\rm Sp}(1)\cdot{\rm GL}(m,\Hq)$ is a Lie group of order two, there exists only one more general
notion of integrability for an almost quaternionic structure, which amounts to the existence of a
compatible torsion free connection (see \cite{Sal-dg_qm}\,). In dimension at least eight,
such a connection is called \emph{quaternionic} whilst, in dimension four, a quaternionic connection
is an anti-self-dual Weyl connection. A \emph{quaternionic manifold} is a manifold endowed with an
almost quaternionic structure and a (compatible) quaternionic connection.\\
\indent
It is a basic fact that the problem of the existence of a quaternionic connection on a manifold,
endowed with an almost quaternionic structure, admits a twistorial interpretation
(see Remark \ref{rem:twist_1-int}(2)\,, below).\\
\indent
In this paper, we introduce a natural notion of \emph{quaternionic map} (Definition \ref{defn:qm}\,)
with respect to which the class of quaternionic manifolds becomes a category. Furthermore, we show
that the quaternionic maps, of rank at least one, are twistorial in a natural way; that is, they are
characterised by the existence of a holomorphic lift between
the corresponding twistor spaces (Theorem \ref{thm:qtwist}\,).\\
\indent
The paper is organised as follows. In Section \ref{section:1} we review some facts on quaternionic
vector spaces (see \cite{AleMar-Annali96}\,). In Section \ref{section:2}\,, after recalling the definition
of almost quaternionic structure, we introduce the notion of quaternionic map and we prove its first
properties (Proposition \ref{prop:firstpropr}\,). Also, in Section \ref{section:2} we recall the
two almost twistorial structures associated to a quaternionic manifold, one of which
(Example \ref{exm:qtwiststr}\,) is integrable, whilst the other one (Example \ref{exm:qtwiststr'}\,)
is nonintegrable.\\
\indent
In Section \ref{section:3} we study twistorial maps between quaternionic manifolds. Besides the
above mentioned relation between quaternionic and twistorial maps, with respect to the
(integrable) twistorial structures, we prove that a map, of rank at least one, is twistorial,
with respect to the nonintegrable almost twistorial structures,
if and only if it is quaternionic and totally geodesic (Theorem \ref{thm:qtwist'}\,).
Another result we obtain is that any quaternionic map is real-analytic, at least, outside the
frontier of the zero set of its differential (Corrolary \ref{cor:q_analytic}\,).\\
\indent
Examples of quaternionic maps are given in Section \ref{section:4}\,. There, we, also, apply
results of Section \ref{section:3} to describe all the quaternionic maps between open sets
of quaternionic projective spaces (Theorem \ref{thm:HP^m}\,).\\
\indent
Finally, in the Appendix we discuss how the quaternionic maps are related to other, more or less similar,
notions. We conclude that the quaternionic maps are the natural morphisms of Quaternionic Geometry.

\section{Quaternionic vector spaces} \label{section:1}

\indent
In this section, we review some facts, from \cite{AleMar-Annali96}\,, on quaternionic vector spaces
and quaternionic linear maps. Unless otherwise stated, all the vector spaces and linear maps are
assumed real.

\begin{defn}
Let $A$ and $B$ be (real or complex, unital) associative algebras. Two morphisms $\r,\s:A\to B$
are called \emph{$A$-equivalent} if there exists an automorphism $\t:A\to A$ such that $\s=\r\circ\t$\,.
\end{defn}

\indent
Let $\Hq$ be the division algebra of quaternions. The group of
automorphisms of $\Hq$ is ${\rm SO}(3)$\,, acting trivially on $1$ and canonically
on ${\rm Im}\Hq(=\R^3)$\,; note that, all the automorphisms of $\Hq$ are interior.\\
\indent
The following definition is due to \cite{AleMar-Annali96}\,.

\begin{defn}
1) A \emph{linear hypercomplex structure} on a vector space $V$ is a morphism of associative algebras
from $\Hq$ to ${\rm End}(V)$\,. A vector space endowed with a linear hypercomplex structure is
called a \emph{hypercomplex vector space}.\\
\indent
2) A \emph{linear quaternionic structure} on a vector space $V$ is an equivalence class
of $\Hq$-equivalent morphisms of associative algebras from $\Hq$ to ${\rm End}(V)$\,.
Any representative of the class defining a linear quaternionic structure is called an
\emph{admissible} linear hypercomplex structure (of the given linear quaternionic structure).
A vector space endowed with a linear quaternionic structure is called a \emph{quaternionic vector space}.
\end{defn}

\indent
Obviously, a hypercomplex vector space is just a left $\Hq$-module.

\begin{exm} \label{exm:qlinear}
The natural structure of left $\Hq$-module on $\Hq^{\!m}$ gives the (natural) linear hypercomplex structure
of $\Hq^{\!m}$, $(m\geq0)$. Moreover, any hypercomplex vector space is $\Hq$-linearly isomorphic to $\Hq^{\!m}$,
for some $m\geq0$\,.\\
\indent
The linear hypercomplex structure of $\Hq^{\!m}$ determines the (natural) linear quaternionic structure of $\Hq^{\!m}$.
\end{exm}

\indent
A \emph{hypercomplex linear map} $f:V\to W$ between hypercomplex vector spaces is an $\Hq$-linear map.\\
\indent
Let $V$ be a quaternionic vector space and let $\r:\Hq\to{\rm End}(V)$ be an admissible linear hypercomplex
structure. As ${\rm SO}(3)$ acting on $\Hq$, preserves $1$ and ${\rm Im}\Hq$, the vector spaces
$Q_V=\r({\rm Im}\Hq)$ and  $\widetilde{Q}_V=\r(\Hq)$ depend only of the linear quaternionic structure induced
by $\rho$ on $V$.
Furthermore, $\widetilde{Q}_V\subseteq{\rm End}(V)$ is a division algebra (noncanonically) isomorphic to $\Hq$
and $Q_V$ is a three-dimensional oriented Euclidean vector space for which any oriented orthonormal basis
$(I,J,K)$ satisfies the quaternionic identities (that is, $I^2=J^2=K^2=IJK=-{\rm Id}_V$\,).
Similarly, the unit sphere $Z_V=\r\bigl(S^2\bigr)$ is well-defined.

\begin{defn}[cf.\ \cite{AleMar-Annali96}\,]
Let\/ $V$ and\/ $W$ be quaternionic vector spaces and let $t:V\to W$ and $T:Z_V\to Z_W$ be maps.\\
\indent
We say that $t$ is a \emph{quaternionic linear map, with respect to $T$,} if $t$ is linear and
$$t\circ J=T(J)\circ t\;,$$ for any $J\in Z_V$\,.
\end{defn}

\begin{prop} \label{prop:qlinearmaps1}
Let $V$ and $W$ be quaternionic vector spaces and let $t:V\to W$ be a nonzero linear map.\\
\indent
{\rm (i)} If $t$ is quaternionic linear, with respect to $T_1,T_2:Z_V\to Z_W$\,, then $T_1=T_2$\,.\\
\indent
{\rm (ii)} If $t$ is quaternionic linear, with respect to some map $T:Z_V\to Z_W$\,, then
$T$ can be uniquely extended to an orientation preserving linear isometry from $Q_V$ to $Q_W$\,.
\end{prop}
\begin{proof}
Let $J\in Z_V$\,. As $t\circ J=T_k(J)\circ t$\,, $(k=1,2)$\,, and $t\neq0$ we have that the kernel of $T_1(J)-T_2(J)$
is nonzero. But $T_1(J)-T_2(J)$ is in $\widetilde{Q}_W$ which is a division algebra. Thus $T_1(J)=T_2(J)$\,.
This proves assertion (i)\,.\\
\indent
To prove (ii) we, firstly, obtain, as above, that if $(I,J,K)$ satisfy the quaternionic identities then, also,
$\bigl(T(I),T(J),T(K)\bigr)$ satisfy the quaternionic identities.\\
\indent
Now, let $(a,b,c)\in S^2$. Then $t\circ(aI+bJ+cK)=T(aI+bJ+cK)\circ t$\,. On the other hand, we have
\begin{equation*}
t\circ(aI+bJ+cK)=a\,t\circ I+b\,t\circ J+c\,t\circ K=\bigl(a\,T(I)+b\,T(J)+c\,T(K)\bigr)\circ t\;.
\end{equation*}
\indent
Thus $T(aI+bJ+cK)\circ t=\bigl(a\,T(I)+b\,T(J)+c\,T(K)\bigr)\circ t$ which, because $t\neq0$\,, implies
that $T(aI+bJ+cK)=a\,T(I)+b\,T(J)+c\,T(K)$\,. The proof follows.
\end{proof}

\indent
Next, we prove the following:

\begin{prop}[\,\cite{AleMar-Annali96}\,] \label{prop:qlinearmaps2}
{\rm (i)} For any quaternionic vector space $V$ there exists a quaternionic linear isomorphism
from $V$ to $\Hq^{\!m}$ (endowed with its natural linear quaternionic structure),
for some $m\geq0$\,\\
\indent
{\rm (ii)} Any quaternionic linear map $t:\Hq^{\!m}\to\Hq^{\!n}$
is given by $t(X)=aXA$\,, $(X\in\Hq^{\!m})$\,, for some $a\in\Hq$ and an $m\times n$ matrix $A$\,,
whose entries are quaternions.
\end{prop}
\begin{proof}
Assertion (i) follows quickly from the fact that any hypercomplex vector space is $\Hq$-linearly isomorphic
to $\Hq^{\!m}$, for some $m\geq0$\,.\\
\indent
Let $t:\Hq^{\!m}\to\Hq^{\!n}$ be a quaternionic linear map, with respect to some map
$T:S^2(=Z_{\Hq^{\!m}})\to S^2(=Z_{\Hq^{\!n}})$.\\
\indent
If $t=0$ then by taking, for example, $a=1$ and $A=0$ assertion (ii) is trivially satisfied.
If $t\neq0$ then, by Proposition \ref{prop:qlinearmaps1}(ii)\,, there exists $a\in{\rm Sp}(1)$
such that $T({\rm i})=a{\rm i}a^{-1}$, $T({\rm j})=a{\rm j}a^{-1}$, $T({\rm k})=a{\rm k}a^{-1}$
and one checks immediately that $t'=a^{-1}t$ is $\Hq$-linear.\\
\indent
Let $A$ be the matrix of $t':\Hq^{\!m}\to\Hq^{\!m}$ with respect to the canonical bases of the
free (left) $\Hq$-modules $\Hq^{\!m}$ and $\Hq^{\!m}$. Then $t:\Hq^{\!m}\to\Hq^{\!n}$
is given by $t(X)=aXA$\,, $(X\in\Hq^{\!m})$\,, and the proof is complete.
\end{proof}

\indent
From Proposition \ref{prop:qlinearmaps2} we obtain the following result.

\begin{cor}[\,\cite{AleMar-Annali96}\,] \label{cor:qlineargroup}
The group of quaternionic linear automorphisms of $\Hq^{\!m}$ is equal to
${\rm Sp}(1)\cdot{\rm GL}(m,\Hq)$\,.
\end{cor}

\indent
Let $V$ be a quaternionic vector space and let $\r:\Hq\to{\rm End}(V)$ be an admissible linear hypercomplex
structure. Obviously, $\r\otimes\r:\Hq\otimes\Hq\to{\rm End}(V\otimes V)$ is also a morphism of associative
algebras (the tensor products are taken over $\R$). As ${\rm SO}(3)$ acts on $\Hq$ by isometries,
$\r\otimes\r$ maps the Euclidean structure
$$1\otimes 1+{\bf i}\otimes{\bf i}+{\bf j}\otimes{\bf j}+{\bf k}\otimes{\bf k}$$
on the (real) dual of $\Hq$ onto an endomorphism $\b$ of
$V\otimes V$ which depends only of the linear quaternionic structure on $V$. Let
$b\in{\rm Hom}(V\otimes V,V\odot V)$ be the composition of $\b$, to the left, with
the projection $V\otimes V\to V\odot V$, where $V\odot V$ is the second symmetric power of $V$.
Note that, $b$ is also characterised by
\begin{equation} \label{e:projector}
b(X,Y)=\tfrac12\sum_{i=0}^3\bigl(E_i(X)\otimes E_i(Y)+E_i(Y)\otimes E_i(X)\bigr)\;,
\end{equation}
for any $X,Y\in V$, where $E_0=\r(1)$\,, $E_1=\r({\rm i})$\,, $E_2=\r({\rm j})$\,, $E_3=\r({\rm k})$\,.

\begin{prop} \label{prop:projector}
Let\/ $V$ be a quaternionic vector space. For any $J\in Z_V$ we denote by $V^{1,0;J}$ and $V^{0,1;J}$
the eigenspaces of $J$ with respect to ${\rm i}$ and $-{\rm i}$\,, respectively.\\
\indent
{\rm (i)} The subspace $b(V\otimes V)$ of $V\odot V$ is equal to the space of Hermitian contravariant
symmetric $2$-forms on $V$ (that is, elements of $V\odot V$ invariant under $J\otimes J$, for any
$J\in Z_V$).\\
\indent
{\rm (ii)} For any $J\in Z_V$ and $\a\in V^*$ we have
$\iota_{\a}\bigl(b(V^{0,1;J},V^{0,1;J})\bigr)=0$.\\
\indent
{\rm (iii)} Let $J,K\in Z_V$ be orthogonal on each other. Then for any $X\in V^{1,0:J}$,
$Y\in V^{0,1:J}$ and $\a\in V^*$ we have
$$\iota_{\a}\bigl(b(X,Y)\bigr)=\a(X)Y+\a(Y)X+\a(KX)KY+\a(KY)KX\;.$$
\end{prop}
\begin{proof}
Assertion (i) follows, for example, from relation \eqref{e:projector}\,.\\
\indent
To prove (ii)\,,
let $\a\in V^*$ and let $I\in Z_V$ be included in any admissible hypercomplex basis
$(I,J,K)$\,. Then
\begin{equation*}
\begin{split}
2\,i_\a(b(X,Y))=&\,\a(X)Y+\a(Y)X+\a(IX)IY+\a(IY)IX\\
             &+\a(JX)JY+\a(JY)JX+\a(KX)KY+\a(KY)KX\;.
\end{split}
\end{equation*}
\indent
If $X,Y \in V^{0,1;I}$ (that is, $IX=-{\bf i}X$, $IY=-{\bf i}Y$) we have
\begin{equation*}
\begin{split}
\a(IX)IY+\a(IY)IX&=-\bigl(\a(X)Y+\a(Y)X\bigr)\;,\\
\a(KX)KY+\a(KY)KX&=\a(IJX)IJY+\a(IJY)IJX\\
                 &=-\bigl(\a(JX)JY+\a(JY)JX\bigr)\;.
\end{split}
\end{equation*}
\indent
The proof of (ii) follows.\\
\indent
Assertion (iii) can be proved similarly.
\end{proof}

\begin{rem}
A result, similar to Proposition \ref{prop:projector}\,, can be straightforwardly established for $\b$\,.
\end{rem}

\section{Quaternionic manifolds and maps} \label{section:2}

\indent
In this section we review some basic facts on (almost) quaternionic manifolds (see \cite{AleMar-Annali96}\,) and we
introduce the notion of quaternionic map.\\
\indent
Unless otherwise stated, all the manifolds and maps are assumed smooth.

\begin{defn}
A \emph{(fibre) bundle of associative algebras} is a vector bundle whose typical fibre
is a (finite-dimensional) associative algebra $A$ and whose structural group is the group of automorphisms
of $A$\,.\\
\indent
Let $E$ and $F$ be bundles of associative algebras. A morphism of vector bundles $\r:E\to F$ is called
a \emph{morphism of bundles of associative algebras} if $\r$ restricted to each fibre is a morphism of
associative algebras.
\end{defn}

\indent
Next, we recall the definitions of almost quaternionic manifolds and almost hypercomplex manifolds.

\begin{defn}[\,\cite{Bon}\,]
An \emph{almost quaternionic structure} on a manifold $M$ is a pair $(E,\r)$ where $E$
is a bundle of associative algebras, over $M$, with typical fibre $\Hq$ and $\r:E\to{\rm End}(TM)$
is a morphism of bundles of associative algebras. An \emph{almost quaternionic manifold}
is a manifold endowed with an almost quaternionic structure.\\
\indent
An \emph{almost hypercomplex structure} on a manifold $M$ is an almost quaternionic
structure $(E,\r)$ for which $E=M\times\Hq$. An \emph{almost hypercomplex manifold}
is a manifold endowed with an almost hypercomplex structure.
\end{defn}

\indent
It is well-known (see \cite{AleMar-Annali96}\,) that there are other ways to define the
almost quaternionic and hypercomplex manifolds.

\begin{prop}
An almost quaternionic structure on a manifold $M$ corresponds to a reduction of the frame bundle of $M$
to ${\rm Sp}(1)\cdot{\rm GL}(m,\Hq)$ (equivalently, to an ${\rm Sp}(1)\cdot{\rm GL}(m,\Hq)$-structure).\\
\indent
An almost hypercomplex structure on a manifold $M$ corresponds to a reduction of the frame bundle of $M$
to ${\rm GL}(m,\Hq)$\,.
\end{prop}
\begin{proof}
Let $M$ be a manifold endowed with an almost quaternionic structure $(E,\r)$\,. At each $x\in M$,
the morphism $\r_x:E_x\to{\rm End}(T_xM)$ determines a structure of quaternionic vector space on $T_xM$.\\
\indent
Let $U$ be an open set of $M$ over which $E$ is trivial. Then, by passing to an open subset, if necessary,
we can construct a local trivialization $h_U:U\times\Hq^{\!m}\to TM|_U$ of $TM$
which induces quaternionic linear isomorphisms on each fibre.\\
\indent
If $h_U$ and $h_V$ are two such local trivializations, with $U\cap V\neq\emptyset$\,,
then, by Corollary \ref{cor:qlineargroup}\,, we have that $\bigl((h_V)^{-1}\circ h_U\bigr)(x,q)=(x,a(x)q)$
for some map $a:U\cap V\to{\rm Sp}(1)\cdot{\rm GL}(m,\Hq)$\,, $(x\in U\cap V,\;q\in\Hq^{\!m})$\,.\\
\indent
Conversely, if the frame bundle of $M$ admits a reduction to ${\rm Sp}(1)\cdot{\rm GL}(m,\Hq)$ then the
morphism of Lie groups ${\rm Sp}(1)\cdot{\rm GL}(m,\Hq)\to{\rm SO}(3)$\,, $a\cdot A\mapsto{\rm Ad}\,a$
determines an oriented Riemannian vector bundle $Q$ of rank three and an injective morphism of vector bundles
$Q\hookrightarrow{\rm End}(TM)$ with the property that any positive local orthonormal frame
of $Q$ satisfies the quaternionic identities. Let $\widetilde{Q}$ be generated by $Q$ and ${\rm Id}_{TM}$.
Then $\widetilde{Q}\hookrightarrow{\rm End}(TM)$ is a subbundle of associative algebras and its typical fibre
is $\Hq$.\\
\indent
The proof for almost hypercomplex manifolds is similar.
\end{proof}

\indent
Let $M$ be a manifold endowed with an almost quaternionic structure $(E,\r)$\,. Then, as each fibre of $E$ is
an associative algebra isomorphic to $\Hq$, there exists an oriented Riemannian vector subbundle of rank three
${\rm Im}E\subseteq E$ with the property that any positive local orthonormal frame of it
satisfies the quaternionic identities. Let $Q_M=\r({\rm Im}E)$ and $\widetilde{Q}_M=\r(E)$\,. Then
$\widetilde{Q}_M\subseteq{\rm End}(TM)$ is a subbundle of associative algebras and its typical fibre
is $\Hq$. Also, $Q_M$ is an oriented Riemannian vector bundle of rank three with the property that any
positive local orthonormal frame of it satisfies the quaternionic identities; denote by $Z_M$ the
sphere bundle of $Q_M$\,.\\
\indent
Note that, any almost quaternionic manifold $M$, $(\dim M=4m)$\,, is oriented; at each $x\in M$,
the orientation of $T_xM$ is given by any $J\in(Z_M)_x$\,. Denote by $L$ the line bundle of $M$;
that is, the line bundle over $M$ associated to the frame bundle of positive frames through the morphism
of Lie groups ${\rm GL}(4m,\R)_0\to(0,\infty)$\,, $a\mapsto(\det a)^{1/(4m)}$\,; see \cite{LouPan-II}\,.
(Sometimes, $(L^*)^{4m}$ is called `the bundle of densities' of $M$ whilst $L$ is called `the bundle of
densities of weight $1$' or, even, `the weight bundle' of $M$.)\\
\indent
Also, as ${\rm Sp}(1)\cdot{\rm GL}(1,\Hq)$ is equal to the connected component of the identity of
${\rm CO}(4)$\,, a four-dimensional almost quaternionic manifold is just an oriented conformal manifold.

\begin{defn} \label{defn:qm}
Let $\phi:M\to N$ be a map between almost quaternionic manifolds and let $\Phi:Z_M\to Z_N$ be such that
$\p_N\circ\Phi=\phi\circ\p_M$, where $\p_M:Z_M\to M$ and $\p_N:Z_N\to N$ are the projections.\\
\indent
Then $\phi$ is a \emph{quaternionic map, with respect to $\Phi$}, if
$\dif\!\phi_{\p_M(J)}\circ J=\Phi(J)\circ\dif\!\phi_{\p_M(J)}$
for any $J\in Z_M$\,.\\
\indent
A \emph{quaternionic immersion/submersion/diffeomorphism} is a quaternionic map which is an
immersion/submersion/diffeomorphism.\\
\indent
An injective quaternionic immersion is called an \emph{almost quaternionic submanifold}.
\end{defn}

\begin{rem}
1) For immersions (and, in particular, diffeomorphisms) our definition of quaternionic map
particularizes to give notions already in use (see, for example, \cite{AleMar-Lincei93}\,).\\
\indent
2) Let $M$, $N$ and $P$ be almost quaternionic manifolds, and let $\phi:M\to N$ and $\psi:N\to P$ be quaternionic
maps, with respect to some maps $\Phi:Z_M\to Z_N$ and $\varPsi:Z_N\to Z_P$\,, respectively. Then, obviously,
$\psi\circ\phi$ is quaternionic, with respect to $\varPsi\circ\Phi$\,.
\end{rem}

\indent
The following result follows quickly from Proposition \ref{prop:qlinearmaps1}\,.

\begin{prop} \label{prop:firstpropr}
Let $M$ and $N$ be almost quaternionic manifolds and let $\phi:M\to N$ be a map of rank at least one.\\
\indent
{\rm (i)} If $\phi$ is quaternionic, with respect to $\Phi_1,\,\Phi_2:Z_M\to Z_N$\,, then $\Phi_1=\Phi_2$\,.\\
\indent
{\rm (ii)} If $\phi$ is quaternionic, with respect to $\Phi:Z_M\to Z_N$\,, then $\Phi$
induces an isomorphism of\/ ${\rm SO}(3)$-bundles\/ $Q_M=\phi^*(Q_N)$\,.
\end{prop}

\indent
Let $M$ be an almost quaternionic manifold, $\dim M=4m$\,. An \emph{almost quaternionic connection} on $M$ is
a connection $\nabla$ which induces a connection on $Q_M$ (that is, if $J$ is a section of $Q_M$ and $X$ is
a vector field on $M$ then $\nabla_XJ$ is a section of $Q_M$); equivalently, $\nabla$ induces a connection on
the reduction to ${\rm Sp}(1)\cdot{\rm GL}(m,\Hq)$ of the frame bundle of $M$, corresponding to the almost
quaternionic structure. If $m\geq2$\,, a \emph{quaternionic connection} on $M$ is a torsion-free almost
quaternionic connection. If $m=1$\,, a \emph{quaternionic connection} on $M$ is an anti-self-dual
Weyl connection.\\

\begin{defn}[\,\cite{Sal-dg_qm}\,; cf.\ \cite{Opr-q77}\,]
A \emph{quaternionic manifold} is an almost quaternionic manifold endowed with a quaternionic connection.
\end{defn}

\indent
The set of quaternionic connections on a quaternionic manifold is well-understood.

\begin{prop}[\,\cite{Opr-q84}\,; see \cite{AleMar-Annali96}\,] \label{prop:qconnections}
Let $M$ be a quaternionic manifold, $\dim M=4m$\,. The set of quaternionic connections on $M$
is an affine space, over the vector space of \mbox{$1$-forms} on $M$, isomorphic to the affine
space of connections on $L$\,: if\/ $\a$ is the difference between the connections induced on
$L^{\frac{2m}{m+1}}$ by two quaternionic connections $\nabla^2$ and\/ $\nabla^1$ on $M$ then
$$\nabla^2_XY=\nabla^1_XY+\a(X)Y+\a(Y)X-\iota_{\a}\bigl(b(X,Y)\bigr)\;,$$
for any vector fields $X$ and\/ $Y$ on $M$.
\end{prop}

\indent
Next, we recall the natural `almost twistorial structures' of a quaternionic manifold
(see \cite{LouPan-II} for the general notion of `almost twistorial structure').

\begin{exm}[\,\cite{Sal-dg_qm}\,] \label{exm:qtwiststr}
Let $M$ be a quaternionic manifold. The quaternionic connection of $M$ induces a connection $\H\subseteq TZ_M$
on $Z_M$\,. Let $\H^{1,0}$ be the complex subbundle of $\H^{\C}$ such that $\dif\!\p_M\bigl(\H^{1,0}_J\bigr)$
is the eigenspace corresponding to ${\rm i}$ of $J\in{\rm End}(T_{\p_M(J)}M)$, for any $J\in Z_M$\,,
where $\p_M:Z_M\to M$ is the projection.\\
\indent
Let $\J_M$ be the almost complex structure on $Z_M$ whose eigenbundle corresponding to ${\rm i}$ is equal
to $\H^{1,0}\oplus({\rm ker}\dif\!\p_M)^{1,0}$.\\
\indent
We have that $\J_M$ is integrable (this can be proved by using \cite[Theorem 1.1]{Pan-tm}\,).
Furthermore, $\J_M$ does not depend of the quaternionic connection on $M$ (this can be proved by using
Propositions \ref{prop:projector}(ii) and \ref{prop:qconnections}\,).\\
\indent
We call $\t_M=(Z_M,M,\p_M,\J_M)$ \emph{the twistorial structure of $M$}.\\
\indent
From the integrability of $\t_M$ it follows that there exists a unique real-analytic structure on $M$
with respect to which the following conditions are satisfied:\\
\indent
\quad{}a) The almost quaternionic structure of $M$ is real-analytic;\\
\indent
\quad{}b) Locally, there exist real-analytic quaternionic connections on $M$.\\
\indent
Note that, the given quaternionic connection on $M$ is not necessarily real-analytic.
\end{exm}

\begin{rem}[cf.\ \cite{AleMarPon-99}\,] \label{rem:twist_1-int}
Let $M$ be an almost quaternionic manifold and let $\nabla$ be an almost quaternionic connection on $M$.
Then, similarly to Example \ref{exm:qtwiststr}\,, we construct an almost complex structure $\J^{\nabla}$
on $Z_M$ and an almost twistorial structure $\t^{\nabla}=(Z_M,M,\p_M,\J^{\nabla})$\,.\\
\indent
1) Let $J$ be a (local) admissible almost complex structure on $M$ and let $s^J$ be the section of $Z_M$
corresponding to $J$. Then any two of the following assertions imply the third:\\
\indent
\quad(i) $J$ is integrable;\\
\indent
\quad(ii) $s^J:(M,J)\to(Z_M,\J^{\nabla})$ is holomorphic;\\
\indent
\quad(iii) $T^{\nabla}\bigl(\Lambda^2(T^{0,1;J}M)\bigr)\subseteq T^{0,1;J}M$, where $T^{\nabla}$ is
the torsion of $\nabla$.\\
\indent
2) \emph{The almost twistorial structure $\t^{\nabla}$ is integrable if and only if $M$ admits a quaternionic
connection and $\t^{\nabla}=\t_M$\,}. Indeed, if $\J^{\nabla}$ is integrable then, locally, there exist many
admissible almost complex structures $J$ which satisfy assertion (ii)\,, above; moreover, as $\J^{\nabla}$ is
integrable, any such $J$ is integrable. If $\dim M=4$\,, it is fairly well-known that, then, any Weyl connection
is anti-self-dual, whilst, if $\dim M\geq8$ then, by \cite[Theorem 2.4]{AleMarPon-99}\,, $M$ admits a quaternionic
connection. Also, we have that $T^{1,0;\J^{\nabla}}\!Z_M$ is, pointwisely, generated by $({\rm ker}\dif\!\p)^{1,0}$
and the holomorphic tangent bundles to the images of the local sections of $Z_M$ corresponding to
admissible local complex structures on $M$. As this, also, holds for any quaternionic connection on $M$,
we obtain $\J^{\nabla}=\J_M$\,.
\end{rem}

\begin{exm}[cf.\ \cite{EelSal}\,] \label{exm:qtwiststr'}
With the same notations as in Example \ref{exm:qtwiststr'}\,, let $\J'_M$ the almost complex structure
on $Z_M$ whose eigenbundle corresponding to ${\rm i}$ is equal to $\H^{1,0}\oplus({\rm ker}\dif\!\p_M)^{0,1}$.\\
\indent
We have that $\J'_M$ is nonintegrable (that is, always not integrable).
Furthermore, from Propositions  \ref{prop:projector}(iii) and \ref{prop:qconnections} it follows
that $\J'_M$ determines the quaternionic connection on $M$.\\
\indent
We call $\t'_M=(Z_M,M,\p_M,\J'_M)$ \emph{the nonintegrable almost twistorial structure of $M$}.
\end{exm}

\indent
We end this section with a well-known fact which will be used later on; for the reader's convenience 
we also supply a proof. 

\begin{prop}[see \cite{AleMar-Report93}\,] \label{prop:qsubmanifold}
Any almost quaternionic submanifold $N$ of a quaternionic manifold $M$ is totally-geodesic
with respect to any quaternionic connection $\nabla$ on $M$. Moreover, $\nabla$ induces
a quaternionic connection on $N$.
\end{prop}
\begin{proof}
We have $J\bigl(T_{\p_M(J)}N\bigr)\subseteq T_{\p_M(J)}N$, for any $J\in(Q_M)|_N$\,,
where $\p_M:Q_M\to M$ is the projection. Thus, the vector bundle $(TM|_N)/TN$ admits a unique reduction
to ${\rm Sp}(1)\cdot{\rm GL}(m-n,\Hq)$ such that the projection $\p:TM\to (TM|_N)/TN$ is
quaternionic linear on each fibre, where $\dim N=4n$ and $\dim M=4m$\,.\\
\indent
Let $B$ be the `second fundamental form' of $N\subseteq M$ with respect to $\nabla$; that is,
$B(X,Y)=\p(\nabla_XY)$, for any vector fields $X$, $Y$ on $N$. We have to prove that $B=0$\,.\\
\indent
Let $J$ be a section of $Z_M$ over some open set $U$ of $M$ which intersects $N$. Then,
for any $x\in U\cap N$ and $X\in T_xN$, we have that $\nabla_XJ\in(Q_M)_x$ and, consequently,
$(\nabla_XJ)(T_xN)\subseteq T_xN$. Therefore, $\nabla_X(JY)-J(\nabla_XY)$ is a vector field on $U\cap N$,
for any vector fields $X$, $Y$ on $U\cap N$.\\
\indent
Hence, for any $J\in(Z_M)|_N$ and $X,Y\in TN$, we have $B(X,JY)=JB(X,Y)$\,; as $B$ is symmetric,
this is equivalent to $B(JX,JY)=-B(X,Y)$\,. By applying this property to a positive orthonormal frame
of $Q_M$ the proof follows quickly.
\end{proof}

\begin{rem}
Proposition \ref{prop:qsubmanifold} motivates the use of the term 'quaternionic submanifold',
instead of 'almost quaternionic submanifold', when dealing with an ambient quaternionic manifold.
\end{rem}

\newpage

\section{Twistorial maps between quaternionic manifolds} \label{section:3}

\indent
The following definition is a particular case of \cite[Definition 4.1]{LouPan-II}\,.

\begin{defn}
Let $M$ and $N$ be quaternionic manifolds and let $\phi:M\to N$ be a map.
Suppose that there exists a map $\Phi:Z_M\to Z_N$ such that $\p_N\circ\Phi=\phi\circ\p_M$\,.\\
\indent
Denote by $\t_M=(Z_M,M,\p_M,\J_M)$ and $\t_N=(Z_N,N,\p_N,\J_N)$ the twistorial structures of $M$ and $N$,
respectively. Then $\phi:(M,\t_M)\to(N,\t_N)$ is a \emph{twistorial map, with respect to $\Phi$},
if $\Phi:(Z_M,\J_M)\to(Z_N,\J_N)$ is holomorphic.\\
\indent
Similarly, $\phi:(M,\t'_M)\to(N,\t'_N)$ is a \emph{twistorial map, with respect to $\Phi$},
if $\Phi:(Z_M,\J'_M)\to(Z_N,\J'_N)$ is holomorphic, where $\t'_M=(Z_M,M,\p_M,\J'_M)$ and
$\t'_N=(Z_N,N,\p_N,\J'_N)$ are the nonintegrable almost twistorial structures of $M$ and $N$,
respectively.
\end{defn}

\indent
Let $M$ be an almost quaternionic manifold and let $J\in Z_M$\,. We denote $T^{1,0;J}M$
and $T^{0,1:J}M$ the eigenspaces of $J$ corresponding to ${\rm i}$ and $-{\rm i}$\,, respectively.

\begin{prop} \label{prop:qtwist}
Let $M$ and $N$ be quaternionic manifolds; denote by\/ $\nabla^M$ and\/ $\nabla^N$ the quaternionic
connections of\/ $M$ and $N$, respectively.\\
\indent
Let $\phi:M\to N$ be a quaternionic map with respect to some map $\Phi:Z_M\to Z_N$\,;
suppose that $\phi$ is of rank at least one.\\
\indent
Then the following assertions are equivalent:\\
\indent
\quad{\rm (i)} $\phi:(M,\t_M)\to(N,\t_N)$ is twistorial, with respect to $\Phi$.\\
\indent
\quad{\rm (ii)} $(\nabla\!\dif\!\phi)(T^{0,1;J}_{\p_M(J)}M,T^{0,1;J}_{\p_M(J)}M)\subseteq T^{0,1;J}_{\phi(\p_M(J))}N$,
for any $J\in Z_M\bigl(=\phi^*(Z_N)\bigr)$\,, where $\nabla$ is the connection,
on ${\rm Hom}\bigl(TM,\phi^*(TN)\bigr)$\,, induced by\/ $\nabla^M$ and\/ $\nabla^N$.\\
\indent
\quad{\rm (iii)} $A(JX)J=J\circ(A(X)J)$, for any $J\in Z_M$ and $X\in T_{\p_M(J)}M$, where $A$ is the difference
between the connections, on $Q_M\bigl(=\phi^*(Q_N)\bigr)$\,, induced by\/ $\nabla^M$ and\/ $\nabla^N$.
\end{prop}
\begin{proof}
To prove (i)$\iff$(iii)\,, note that it is sufficient to consider $\phi$ of constant rank
(there exists a dense open subset of $M$ such that on each of its connected components $\phi$ has constant rank).
As then, locally, the image of $\phi$ is a quaternionic submanifold of $N$, Proposition \ref{prop:qsubmanifold}
implies that we can further assume $\phi$ submersive.\\
\indent
Let $J_0\in Z_M$ and let $x_0=\p_M(J_0)$\,. Also, let $S$ be (the image of) a local section of $\phi$\,,
containing $x_0$\,, such that $T_{x_0}S$ is preserved by $J_0$\,. Locally, we may extend $J_0$ to a section
$J$ of $Z_M$\,, over $S$\,. Then there exists a local section $\check{J}$ of $Z_N$ such that
$J=\phi^*(\check{J}\,)$ (equivalently, $\Phi\circ J=\check{J}\circ\phi$\,);
denote $\check{J}_0=\check{J}_{\phi(x_0)}$\,.\\
\indent
Now, the differential of $J$ at $x_0$ is a complex linear map from $\bigl(T_{x_0}S,J_0|_{T_{x_0}S}\bigr)$
to $\bigl(T_{J_0}Z_M,(\J_M)_{J_0}\bigr)$ if and only if $\nabla^M_{J_0X}J=J_0\circ\nabla^M_XJ$,
for any $X\in T_{x_0}S$\,. Similarly, the differential of $\check{J}$ at $\phi(x_0)$ is a
complex linear map from $\bigl(T_{\phi(x_0)}N,\check{J}_0\bigr)$ to
$\bigl(T_{\check{J}_0}Z_N,(\J_N)_{\check{J}_0}\bigr)$ if and only if
$\nabla^N_{\check{J}_0X}\check{J}=\check{J}_0\circ\nabla^N_X\check{J}$, for any $X\in T_{\phi(x_0)}N$\,.\\
\indent
It follows that $\dif\!\Phi_{J_0}:\bigl(T_{J_0}Z_M,(\J_M)_{J_0}\bigr)\to
\bigl(T_{\check{J}_0}Z_N,(\J_N)_{\check{J}_0}\bigr)$ is complex linear if and only if
$A(J_0X)J_0=J_0\circ(A(X)J_0)$\,, for any $X\in T_{x_0}M$.\\
\indent
To prove (ii)$\iff$(iii)\,, firstly, note that, by considering $\dif\!\phi$ as a section of
${\rm Hom}\bigl(TM,\phi^*(TN)\bigr)$\,, we have $\dif\!\phi\circ J=J\circ\dif\!\phi$\,,
for any $J\in Z_M$\,. By taking the covariant derivative of this equality we obtain
\begin{equation*}
(\nabla_X\!\dif\!\phi)\circ J-J\circ(\nabla_X\!\dif\!\phi)=(A(X)J)\circ\dif\!\phi\;,
\end{equation*}
for any $J\in Z_M$ and $X\in T_{\p_M(J)}M$. The proof follows.
\end{proof}

\begin{rem}
The method of Proposition \ref{prop:qtwist} can be applied in several other contexts.
For example, let $(M^m,c,D)$ be a Weyl space, $\dim M=m$\,, and
let $1\leq r\leq\tfrac12\,m$\,. If $r<\tfrac12m$ let $\p:P\to M$ be the bundle of skew-adjoint $f$-structures
on $(M^m,c)$ whose kernels have dimension $m-2r$ (any $F\in P$ is a skew-adjoint linear map on
$\bigl(T_{\p(F)}M,c_{\p(F)}\bigr)$ such that $F^3+F=0$ and $\dim({\rm ker}F)=m-2r$\,).
If $r=\tfrac12\,m$\,, ($m$ even)\,, let $P$ be the bundle of positive orthogonal complex structures on $(M^m,c)$\,.\\
\indent
Then $D$ induces a connection $\H\subseteq TP$ on $P$. Define $\H^0\subseteq\H$ such that
$\dif\!\p(\H^0_F)=T^{0;F}_{\p(F)}M$,
where $T^{0;F}_{\p(F)}M$ is the eigenspace of $F$ corresponding to $0$\,, $(F\in P)$\,.
Also, define $\H^{0,1}\subseteq\H^{\C}$ such that $\dif\!\p(\H^{0,1}_F)=T^{0,1;F}_{\p(F)}M$,
where $T^{0,1;F}_{\p(F)}M$ is the eigenspace of $F$ corresponding to $-{\rm i}$\,, $(F\in P)$\,.\\
\indent
Let $\Fa$ be the almost $f$-structure on $P$ whose eigendistributions corresponding to $0$ and $-{\rm i}$
are $\H^0$ and $\H^{0,1}\oplus({\rm ker}\dif\!\p_M)^{0,1}$, respectively. Then
$\t_{m,r}=(P,M,\p,\Fa)$ is an almost twistorial structure on $M$; see \cite{Pan-tm} for
the characterisation of the integrability of $\t_{m,r}$\,.\\
\indent
Now, let $(M^{2n},c_M,D^M)$ and $(N^{2n-1},c_N,D^N)$ be Weyl spaces; denote by
$\t^M_{2n,n}=(P_M,M,\p_M,\J)$ and $\t^N_{2n-1,n-1}=(P_N,N,\p_N,\Fa)$ the corresponding
almost twistorial structures.\\
\indent
Let $\phi:(M^{2n},c_M)\to(N^{2n-1},c_N)$ be a horizontally conformal submersion. There exists
a unique map $\Phi:P_M\to P_N$ such that $\p_N\circ\Phi=\phi\circ\p_M$ and
$$\dif\!\phi\bigl(T^{0,1;J}_{\p_M(J)}M\bigr)=T^{0;\Phi(J)}_{\p_N(\Phi(J))}N\oplus T^{0,1;\Phi(J)}_{\p_N(\Phi(J))}N\;,$$
for any $J\in P_M$\,.\\
\indent
The following assertions are equivalent:\\
\indent
\quad(i) $\phi:(M^{2n},\t^M_{2n,n})\to(N^{2n-1},\t^N_{2n-1,n-1})$ is twistorial, with respect to $\Phi$
(that is, $\Phi:(P_M,\J)\to(P_N,\Fa)$ is holomorphic; equivalently,
$\dif\!\Phi(T^{0,1}P_M)\subseteq T^0P_N\oplus T^{0,1}P_N$).\\
\indent
\quad(ii) $(D\!\dif\!\phi)(T^{0,1;J}_{\p_M(J)}M,T^{0,1;J}_{\p_M(J)}M)\subseteq
\dif\!\phi(T^{0,1;J}_{\p_M(J)}M)$, for any $J\in P_M$\,, where $D$ is the connection,
on ${\rm Hom}\bigl(TM,\phi^*(TN)\bigr)$\,, induced by $D^M$ and $D^N$.\\
\indent
In the particular case $n=2$\,, from the above equivalence it follows quickly the known
(see \cite{LouPan-II} and the references therein) characterisation of twistorial maps from
four-dimensional conformal manifolds to three-dimensional Weyl spaces; also, recall that then,
if (i) or (ii) holds, $(M^4,c_M)$ is anti-self-dual if and only if $(N^3,c_N,D^N)$ is Einstein--Weyl.\\
\indent
If $n\geq3$ and (i) or (ii) holds then $(M^{2n},c_M)$ is flat if and only if
$D^N$ is the Levi-Civita connection of constant curvature local representatives of $c_N$\,.
\end{rem}

\indent
Similarly to Proposition \ref{prop:qtwist}\,, we obtain the following result.

\begin{prop} \label{prop:qtwist'}
Let $M$ and $N$ be quaternionic manifolds; denote by\/ $\nabla^M$ and\/ $\nabla^N$ the quaternionic
connections of\/ $M$ and $N$, respectively.\\
\indent
Let $\phi:M\to N$ be a quaternionic map with respect to some map $\Phi:Z_M\to Z_N$\,;
suppose that $\phi$ is of rank at least one.\\
\indent
Then the following assertions are equivalent:\\
\indent
\quad{\rm (i)} $\phi:(M,\t'_M)\to(N,\t'_N)$ is twistorial, with respect to $\Phi$.\\
\indent
\quad{\rm (ii)} $(\nabla\!\dif\!\phi)(T^{1,0;J}_{\p_M(J)}M,T^{0,1;J}_{\p_M(J)}M)\subseteq T^{0,1;J}_{\phi(\p_M(J))}N$,
for any $J\in Z_M\bigl(=\phi^*(Z_N)\bigr)$\,, where $\nabla$ is the connection,
on ${\rm Hom}\bigl(TM,\phi^*(TN)\bigr)$\,, induced by\/ $\nabla^M$ and\/ $\nabla^N$.\\
\indent
\quad{\rm (iii)} $A(JX)J=-J\circ(A(X)J)$, for any $J\in Z_M$ and $X\in T_{\p_M(J)}M$, where $A$ is the difference
between the connections, on $Q_M\bigl(=\phi^*(Q_N)\bigr)$\,, induced by\/ $\nabla^M$ and\/ $\nabla^N$.
\end{prop}

\indent
Next, we prove the following result.

\begin{thm} \label{thm:qtwist}
Let $\phi:M\to N$ be a map between quaternionic manifolds and let
$\Phi:Z_M\to Z_N$ be such that $\p_N\circ\Phi=\phi\circ\p_M$\,.\\
\indent
If the zero set of the differential of $\phi$ has empty interior then the following assertions are equivalent:\\
\indent
\quad{\rm (i)} $\phi:(M,\t_M)\to(N,\t_N)$ is twistorial, with respect to $\Phi$.\\
\indent
\quad{\rm (ii)} $\phi:M\to N$ is quaternionic, with respect to $\Phi$.
\end{thm}
\begin{proof}
Obviously, (i)$\Longrightarrow$(ii)\,. Thus, it is sufficient to prove (ii)$\Longrightarrow$(i)\,.\\
\indent
Let $F\subseteq M$ be the zero set of the differential of $\phi$\,. As $M\setminus F$ is dense in $M$
and $\p_M$ is open, we have $\p_M^{-1}(M\setminus F)$ dense in $Z_M$\,. Thus, we may suppose that, at each point,
$\phi$ has rank at least one.\\
\indent
By Proposition \ref{prop:qtwist}\,, it is sufficient to prove that if $\phi$ is
quaternionic then
$$(\nabla\!\dif\!\phi)(T^{0,1;J}_{\p_M(J)}M,T^{0,1;J}_{\p_M(J)}M)\subseteq T^{0,1;J}_{\phi(\p_M(J))}N\;,$$
for any $J\in Z_M\bigl(=\phi^*(Z_N)\bigr)$\,.\\
\indent
As in the proof of Proposition \ref{prop:qtwist}\,, we may suppose $\phi$ submersive. Denote $\V={\rm ker}\dif\!\phi$
and let $\H$ be a distribution on $M$, complementary to $\V$, and which is preserved by $Z_M$ (for example, let
$\H$ be the orthogonal complement of $\V$ with respect to some Hermitian metric on $M$); as usual, we identify
$\H=\phi^*(TN)$\,.\\
\indent
Let $J$ be an admissible complex structure (locally) defined on $M$. As $\V$ and $\H$ are invariant under $J$,
we have decompositions $\V^{\C}=\V^{0,1;J}\oplus\V^{1,0;J}$ and $\H^{\C}=\H^{0,1;J}\oplus\H^{1,0;J}$.\\
\indent
Because $\phi$ is quaternionic, we have $\H^{0,1;J}=\phi^*(T^{0,1;J}N)$\,.\\
\indent
Let $V$ and $X$ be sections of $\V^{0,1;J}$ and $T^{0,1;J}M$, respectively. As the section of $Z_M$ corresponding
to $J$ is a holomorphic map from $(M,J)$ to $(Z_M,\J_M)$\,, we have that $\nabla^M_{\,V}X$ is a section of $T^{0,1;J}M$.
Hence, $(\nabla\!\dif\!\phi)(V,X)=-\dif\!\phi(\nabla^M_{\,V}X)$ is a section of $\H^{0,1;J}$\,.\\
\indent
{}From the fact that $\nabla^M$ and $\nabla^N$ are torsion-free it follows that there exists a
section $\a$ of $\H^*$ such that $(\nabla\!\dif\!\phi)(X,Y)=S^{\a}(X,Y)$\,, for any $X,Y\in\H$
(cf.\ Proposition \ref{prop:qconnections}\,). Hence, by Proposition \ref{prop:projector}\,,
$(\nabla\!\dif\!\phi)(\H^{0,1;J},\H^{0,1;J})\subseteq\H^{0,1;J}$.\\
\indent
The proof is complete.
\end{proof}

\indent
{}From Theorem \ref{thm:qtwist} we obtain the following result (which, also, holds for a more general
class of twistorial maps).

\begin{cor} \label{cor:q_analytic}
Any quaternionic map between quaternionic manifolds is real-analytic, at least, outside the frontier
of the zero set of its differential.
\end{cor}

\indent
A condition equivalent to assertion (ii) of the following result is used in \cite{IanMazVil}\,,
for maps between quaternionic K\"ahler manifolds.

\begin{thm} \label{thm:qtwist'}
Let $\phi:M\to N$ be a map between quaternionic manifolds and let
$\Phi:Z_M\to Z_N$ be such that $\p_N\circ\Phi=\phi\circ\p_M$\,.\\
\indent
If the zero set of the differential of $\phi$ has empty interior then the following assertions are equivalent:\\
\indent
\quad{\rm (i)} $\phi:(M,\t_M')\to(N,\t_N')$ is twistorial, with respect to $\Phi$.\\
\indent
\quad{\rm (ii)} $\phi$ is quaternionic, with respect to $\Phi$, and the connections on \mbox{$Q_M(=\phi^*(Q_N))$,}
induced by the quaternionic connections of $M$ and $N$, are equal.\\
\indent
\quad{\rm (iii)} $\phi$ is a totally geodesic map which is quaternionic, with respect to $\Phi$.
\end{thm}
\begin{proof}
This can be proved as follows. By Proposition \ref{prop:qtwist'}\,, we have
(ii)$\Longrightarrow$(i) and (iii)$\Longrightarrow$(i). Thus, it remains to prove
(i)$\Longrightarrow$(ii),(iii).\\
\indent
If (i) holds then $\phi$ is quaternionic and, by Theorem \ref{thm:qtwist}\,,
$\phi:(M,\t_M)\to(N,\t_N)$ is twistorial. Thus, assertion (iii) of Proposition \ref{prop:qtwist} and
assertion (iii) of Proposition \ref{prop:qtwist'} are both satisfied. This shows that
(i)$\Longrightarrow$(ii)\,.\\
\indent
Finally, from Proposition \ref{prop:qtwist'} it follows quickly that the (1,1)-component of
$\nabla\!\dif\!\phi$ is zero. This implies $\nabla\!\dif\!\phi=0$ and the proof is complete.
\end{proof}

\indent
Next, we explain why, in Theorems \ref{thm:qtwist} and \ref{thm:qtwist'}\,, the assumption on the zero
set of the differential of the map cannot be weakened.

\begin{rem} \label{rem:qtconstant}
Let $M$ and $N$ be quaternionic manifolds. Suppose that $Z_M=M\times S^2$ is the trivial bundle
and let $\p:Z_M\to S^2$ be the projection.\\
\indent
Let $y\in N$ and let $T:S^2\to(Z_N)_y$ be an orientation preserving isometry. Then the constant map
$\phi:M\to N$\,, $x\mapsto y$\,, $(x\in M)$\,, is quaternionic, with respect to $\Phi=T\circ\p$\,.\\
\indent
On the other hand, $\phi$ is twistorial, with respect to $\Phi$, if and only if $M$ is hypercomplex.
\end{rem}

\indent
Note that, Theorems \ref{thm:qtwist} and \ref{thm:qtwist'} hold for any nonconstant real-analytic map
(without any assumption on the zero set of the differential of the map).

\section{Examples and further results} \label{section:4}

\indent
Firstly, we mention that, as any quaternionic submanifold corresponds to an injective
quaternionic immersion, in \cite{Tas-symqsubm} can be found many examples of quaternionic maps.
For example, we have the following:

\begin{exm}
Let ${\rm Gr}_2(m+2,\C)$ be the Grassmannian manifold of complex vector
subspaces, of dimension $2$\,, of $\C^{\!m+2}$, $(m\geq1)$\,.\\
\indent
This is a quaternionic manifold of (real) dimension $4m$\,.
Its twistor space is the flag manifold ${\rm F}_{1,m+1}(m+2,\C)$ of pairs $(l,p)$ with $l$ and $p$
complex vector subspaces of $\C^{\!m+2}$ of dimensions $1$ and $m+1$\,, respectively, such that
$l\subseteq p$\,. The projection $\p:{\rm F}_{1,m+1}(m+2,\C)\to{\rm Gr}_2(m+2,\C)$ is defined
by $\p(l,p)=l\oplus p^{\perp}$, for any $(l,p)\in{\rm F}_{1,m+1}(m+2,\C)$\,, where $p^{\perp}$
is the orthogonal complement of $p$ with respect to the canonical Hermitian product on $\C^{\!m+2}$.\\
\indent
Any injective complex linear map $A:\C^{\!m+2}\to\C^{\!n+2}$\,, $(1\leq m\leq n)$\,,
induces, canonically, a quaternionic map $\phi^A:{\rm Gr}_2(m+2,\C)\to{\rm Gr}_2(n+2,\C)$\,, with respect
to the map $\Phi^A:{\rm F}_{1,m+1}(m+2,\C)\to{\rm F}_{1,n+1}(n+2,\C)$ defined by
$\Phi^A(l,p)=\bigl(A(l),A(p)\oplus q\bigr)$\,,
for any $(l,p)\in{\rm F}_{1,m+1}(m+2,\C)$\,, where $q\subseteq\C^{\!n+2}$ is a fixed complement
of ${\rm im}A$ in $\C^{\!n+2}$. (Note that, if we choose another complement of ${\rm im}A$ in $\C^{\!n+2}$ then
$\Phi^A$ changes by a composition, to the left, with a holomorphic diffeomorphism of
${\rm F}_{1,n+1}(n+2,\C)$\,.)
\end{exm}

\indent
The next example shows that, besides quaternionic immersions, there are many other quaternionic maps.

\begin{exm} \label{exm:HP^m}
Let $\Hq\!P^m$ be the left quaternionic projective space of (real) dimension $4m$\,, $(m\geq1)$\,.
This is a quaternionic manifold (see \cite{Mar-70}\,).
Its twistor space is $\C\!P^{2m+1}$, where the projection $\p:\C\!P^{2m+1}\to\Hq\!P^m$ is
induced by the identification $\Hq^{\!m+1}=\C^{\!2m+2}$, through the morphism of Lie groups
$\C^{\!*}\to\Hq^{\!*}$\,.\\
\indent
Let $A:\Hq^{\!m+1}\to\Hq^{\!n+1}$ be a left $\Hq$-linear map, $(m,n\geq1)$\,. Then
$A$ induces two maps $\phi^A:\Hq\!P^m\setminus P_{\,\Hq}({\rm ker}A)\to \Hq\!P^n$ and
$\Phi^A:\C\!P^{2m+1}\setminus P_{\C}({\rm ker}A)\to\C\!P^{2n+1}$, where $P_{\,\Hq}({\rm ker}A)$
and $P_{\C}({\rm ker}A)$ are the quaternionic and complex projective spaces, respectively,
determined by ${\rm ker}A$\,.\\
\indent
Then $\phi^A$ is a quaternionic map, with respect to $\Phi^A$ (just note, for example, that $\Phi^A$ is
holomorphic).
\end{exm}

\begin{exm} \label{exm:q-Hopf}
Let $\p:\Hq^{\!m+1}\setminus\{0\}\to\Hq\!P^m$ be the Hopf fibration. Then $\p$ is a quaternionic map,
with respect to the canonical projection
$$\Pi:2(m+1)\mathcal{O}(1)\setminus0\to P\bigl(2(m+1)\mathcal{O}(1)\bigr)=\C\!P^{2m+1}\;,$$
where $\mathcal{O}(1)$ is the dual of the tautological line bundle over $\C\!P^1$.
\end{exm}

\indent
The following example, based on a construction of \cite{Swa-bundle}\,, is, essentially,
a generalization of Example \ref{exm:q-Hopf}\,.

\begin{exm}
Let $M$ be a quaternionic manifold of dimension $4m$\,, $(m\geq1)$\,, and
let $\bigl(P,M,{\rm Sp}(1)\cdot{\rm GL}(m,\Hq)\bigr)$ be its bundle of quaternionic frames
(that is, quaternionic linear isomorphisms from $\Hq^{\!m}$ to $T_xM$, $(x\in M)$\,).\\
\indent
Define $\rho:{\rm Sp}(1)\cdot{\rm GL}(m,\Hq)\to\Hq^{\!*}/\{\pm1\}$ by $\rho(a\cdot A)=\pm a$\,,
for any $a\cdot A\in {\rm Sp}(1)\cdot{\rm GL}(m,\Hq)$\,, where $\Hq^{\!*}=\Hq\setminus\{0\}$\,.\\
\indent
Denote $E=\rho(P)$\,. Then $\bigl(E,M,\Hq^{\!*}/\{\pm1\}\bigr)$ is a principal bundle. Furthermore,
the quaternionic connection of $M$ induces a principal connection $\H\subseteq TE$\,.
Let $\V={\rm ker}\,\p$\,, where $\p:E\to M$ is the projection.\\
\indent
Let $q\in S^2\subseteq{\rm Im}\Hq$\,. The multiplication to the right by $-q$ defines
a negative (linear) complex structure on $\Hq$ which, obviously, is invariant under the left action
of $\Hq^{\!*}$ on $\Hq$. Thus, $q$ induces on $\V$ a structure of complex vector bundle $J^{q,\V}$.\\
\indent
As $\widetilde{Q}_M$ is a bundle associated to $E$\,, we have
$\p^*(\widetilde{Q}_M)=E\times\Hq$\,. Together
with the fact that $\H=\p^*(TM)$ this induces a left action of $\Hq^{\!*}$ on $\H$. In particular,
$q$ induces on $\H$ a structure of complex vector bundle $J^{q,\H}$.\\
\indent
Obviously, $J^q=J^{q,\V}\oplus J^{q,\H}$ is an almost complex structure on $E$\,.\\
\indent
The morphism of Lie groups $\C^{\!*}\to\Hq^{\!*}$ given by
$a+b{\rm i}\mapsto a+bq$\,, $(a,b\in\R)$\,, induces a right action of $\C^{\!*}/\{\pm1\}$ on $E$\,.
Furthermore, the quotient of $E$ through this action is $Z_M$ and the projection
$\psi^q:(E,J^q)\to(Z_M,\J^M)$ is holomorphic.\\
\indent
From the fact that $\J^M$ is integrable and the holonomy group of $\H$ is contained by ${\rm Sp}(1)$
it follows that $J^q$ is integrable. Thus, $(J^{\rm i},J^{\rm j},J^{\rm k})$ defines a hypercomplex
structure on $E$\,. The complex structure $\J_E$ of its twistor space $Z_E(=E\times S^2)$ is characterised
by the following: $\J_E|_{\{e\}\times S^2}$ is the canonical complex structure of $S^2$ whilst
$\J_E|_{E\times\{q\}}=J^q$, for any $e\in E$ and $q\in S^2$\,.\\
\indent
Note that, $Z_M=E\times_{\chi}S^2$, where $\chi:\bigl(\Hq^{\!*}/\{\pm1\}\bigr)\times S^2\to S^2$
is defined by $\chi(\pm p,q)=pqp^{-1}$,
for any $\pm p\!\in\!\Hq^{\!*}/\{\pm1\}$ and $q\in S^2$. Let $\Pi:Z_E\to Z_M$ be the projection. Alternatively,
$\Pi$ can be defined by $\Pi(e,q)=\psi^q(e)$\,, for any $(e,q)\in Z_E$\,.\\
\indent
Then $\Pi:(Z_E,\J_E)\to(Z_M,\J_M)$ is holomorphic and $\p:E\to M$ is a quaternionic submersion,
with respect to $\Pi$\,.
\end{exm}

\indent
Next, we prove that Example \ref{exm:HP^m} gives all the quaternionic maps between open sets
of quaternionic projective spaces.

\begin{thm} \label{thm:HP^m}
Let $U$ be a connected open set of\/ $\Hq\!P^m$ and let $\phi:U\to\Hq\!P^n$ be a quaternionic map, $(m,n\geq1)$\,.\\
\indent
Then there exists an $\Hq$\!-linear map $A:\Hq^{\!m+1}\to\Hq^{\!n+1}$ such that $\phi=\phi^A|_U$ and, in
particular, $U\cap P_{\,\Hq}({\rm ker}A)=\emptyset$\,.
\end{thm}

To prove Theorem \ref{thm:HP^m}\,, we need to lemmas, the first of which is, most likely,
known but we do not have a reference for it.

\begin{lem} \label{lem:1HP^m}
Let $M$ be a quaternionic manifold and let $\t_M=(Z_M,M,\p_M,\J_M)$ be its twistorial structure.\\
\indent
Then a function $f:M\to\C$ is constant if and only if $f\circ\p_M:(Z_M,\J_M)\to\C$
is holomorphic.
\end{lem}
\begin{proof}
If $f\circ\p_M:(Z_M,\J_M)\to\C$ is holomorphic then $f$ is holomorphic with respect to any
(local) admissible complex structure $J$ on $M$; equivalently, the differential of $f$ is zero
on $T^{0,1;J}M$.\\
\indent
It follows that $\dif\!f=0$ and the lemma is proved.
\end{proof}

\begin{lem}[cf.\ \cite{AleMar-Annali96}\,] \label{lem:2HP^m}
Let $A$ be a complex linear map from $\Hq^{\!m+1}=\C^{\!2m+2}$ to $\Hq^{\!n+1}=\C^{\!2n+2}$\,,
of complex rank at least $4$\,, $(m,n\geq1)$\,.
Suppose that there exists an open set $U\subseteq\Hq^{\!m+1}$ such that $A$ maps the intersection
of any quaternionic line (through the origin) with $U$ into a quaternionic line.\\
\indent
Then $A$ is quaternionic linear.
\end{lem}
\begin{proof}
The (germ unique) complexification of $\Hq\!P^m$ is ${\rm Gr}_2(2m+2,\C)$ (to show this,
use the fact that ${\rm GL}(m+1,\Hq)$ acts transitively on $\Hq\!P^m$).\\
\indent
Clearly, $A$ determines a holomophic map $\psi^A:\mathfrak{A}\to(\Hq\!P^n)^{\C}$,
where $\mathfrak{A}$ is the (open) subset of ${\rm Gr}_2(2m+2,\C)$ formed of those
two-dimensional complex vector subspaces of $\C^{\!2m+2}$ whose intersection with ${\rm ker}A$ is $\{0\}$\,.
Note that, $\psi^A$ has a pole along $(\Hq\!P^m)^{\C}\setminus\mathfrak{A}$\,.\\
\indent
From the hypothesis, it follows that the restriction of $\psi^A$ to some open set intertwines
the conjugations $C_m$ and $C_n$ of $(\Hq\!P^m)^{\C}$ and $(\Hq\!P^n)^{\C}$, respectively.
By analyticity, we obtain that $\psi^A$ and $C_n\circ\psi^A\circ C_m$ determine
a holomorphic map on $\mathfrak{A}\cup C_m(\mathfrak{A})$\,. Consequently, $C_m(\mathfrak{A})=\mathfrak{A}$
(otherwise, $\psi^A$ could be holomorphically extended over points of $(\Hq\!P^m)^{\C}\setminus\mathfrak{A}$)
and $C_n\circ\psi^A=\psi^A\circ C_m$\,. Hence, $A$ maps any quaternionic line whose intersection
with ${\rm ker}A$ is $\{0\}$ onto a quaternionic line.\\
\indent
Note that, there are no quaternionic lines which intersect ${\rm ker}A$ along complex vector
spaces of dimension $1$ (otherwise, the map $\Phi^A$ from $\C\!P^{2m+1}\setminus P_{\C}({\rm ker}A)$ to
$\C\!P^{2n+1}$, determined by $A$\,, would induce a continuous extension of $\psi^A$ over points of
$(\Hq\!P^m)^{\C}\setminus\mathfrak{A}$).\\
\indent
Thus, $A$ maps any quaternionic line into a quaternionic line and, by \cite[Theorem 1.1]{AleMar-Annali96}\,,
the proof of the lemma is complete.
\end{proof}

\begin{rem}
With the same notations as in Lemma \ref{lem:2HP^m} and its proof, let
$\phi^A:\Hq\!P^m\setminus P_{\,\Hq}({\rm ker}A)\to \Hq\!P^n$ be the quaternionic map
determined by $A$. Then $\psi^A$ is the complexification of $\phi^A$.
\end{rem}

\begin{proof}[Proof of Theorem \ref{thm:HP^m}]
Let $\phi$ be a quaternionic map from a connected open set of $\Hq\!P^m$ to $\Hq\!P^n$.\\
\indent
We, firstly, assume the differential of $\phi$ nowhere zero.\\
\indent
Then, by Corollary \ref{cor:q_analytic}\,, $\phi$ is real-analytic.
Therefore, it is sufficient to find an $\Hq$-linear map $A$ such that $\phi=\phi^A$ on
some open set. Moreover, similarly to the proof of Proposition \ref{prop:qtwist}\,,
we may assume $\phi$ submersive, surjective and with connected fibres.\\
\indent
By Theorem \ref{thm:qtwist}\,, $\phi$ is twistorial, with respect to some holomorphic map $\Phi$
between open sets $U$ and $V$ of $\C\!P^{2m+1}$ and $\C\!P^{2n+1}$, respectively. Also,
$U$ and $V$ contain families of projective lines (the twistor lines) which are
mapped one onto another by $\Phi$\,. Moreover, as the complexification of
$\Hq\!P^m$ contains many complex-quaternionic submanifolds, the preimage through $\Phi$
of a hyperplane of $\C\!P^{2n+1}$ (not disjoint from $V$) is an open
subset of a hyperplane of $\C\!P^{2m+1}$.\\
\indent
An argument similar to the one used in \cite[page 65]{GriHar-Principles}
shows that the nonhomogeneous components of $\Phi$ divided by suitable linear
functions are constant along the twistor lines. Hence, by Lemma \ref{lem:1HP^m}\,,
these meromorphic functions are constant.\\
\indent
We have thus proved that $\Phi$ is induced by some complex linear map
from $\C^{\!2m+2}$ to $\C^{\!2n+2}$. Then the proof (under the assumption that the differential
of $\phi$ is nowhere zero) follows from Lemma \ref{lem:2HP^m}\,.\\
\indent
Finally, note that, if an $\Hq$\!-linear map $A:\Hq^{\!m+1}\to\Hq^{\!n+1}$ induces a nonconstant
(quaternionic) map $\phi^A:\Hq\!P^m\setminus P_{\,\Hq}({\rm ker}A)\to \Hq\!P^n$ then its real rank
is, at least, eight. Hence, at each point, the rank of the differential of $\phi^A$ is at least four.
It follows quickly that, the differential of any nonconstant quaternionic map, between open sets
of quaternionic projective spaces, is nowhere zero. The proof is complete.
\end{proof}

\begin{rem}
Theorem \ref{thm:HP^m} can be, also, proved by applying an inductive
argument, based on a result of \cite{Mar-70}\,, to show that in terms of non-homogeneous
quaternionic projective coordinates $(x_{\!j})_{j=1,\dots,m}$\,, and $(y_{\a})_{\a=1,\dots,n}$\,,
on $\mathbb HP^m$ and $\mathbb  HP^n$, respectively, any quaternionic map
$\phi:\Hq\!P^m\to\Hq\!P^n$ is given by
$$y_{\alpha}=(x_{\!j}a^j_0+a^0_0)^{-1}(x_{\!j}a^j_{\a}+a^0_{\a})\;,\qquad(\alpha=1,\dots,n)\;,$$
where the coefficients $a_{\a}^j$ are constant quaternions (and the Einstein summation convention
is used).
\end{rem}

\indent
We end this section with the following immediate consequence of Theorem \ref{thm:HP^m}\,.

\begin{cor} \label{cor:HP^m}
Any (globally defined) quaternionic map from $\Hq\!P^m$ to $\Hq\!P^n$ is induced
by an injective $\Hq$\!-linear map $\Hq^{\!m+1}\to\Hq^{\!n+1}$; in particular, $m\leq n$\,.
\end{cor}

\appendix

\section{Comparison with other notions of quaternionicity}

\indent
Firstly, we mention the fairly standard notion of `hypercomplex (triholomorphic) map' between
almost hypercomplex manifolds. Obviously, any such map is quaternionic with
respect to the induced almost quaternionic structures on its domain and codomain.\\
\indent
Secondly, there have been studied maps, between quaternionic K\"ahler manifolds $M$ and $N$,
which pull-back the K\"ahler forms of elements of $Z_N$ to K\"ahler forms of elements of $Z_M$ (see \cite{LiZha}\,).
As the K\"ahler forms are nondegenerate, this condition applies only to immersions and to constant
maps, and, therefore, it is too restrictive (also, the presence of a Riemannian metric is required).\\
\indent
Thirdly, there exists the notion of `regular quaternionic function', of one quaternionic variable,
introduced in \cite{Fue} (see \cite{Sud} for a modern presentation and further results)
and later generalised to maps between hyper-K\"ahler manifolds (see \cite{CheLi}\,,
\cite{Hay} and the references therein):

\begin{defn} \label{defn:qFueter}
Let $V$ and $W$ be quaternionic vector spaces and let \mbox{$T:Z_V\to Z_W$} be an orientation preserving isometry.\\
\indent
We say that a map $t:V\to W$ is \emph{linear Fueter-quaternionic, with respect to $T$},
if $t$ is real linear and for some (and, consequently, any) positive orthonormal basis $(I,J,K)$ of $Q_V$
we have $t=T(I)\circ t\circ I+T(J)\circ t\circ J+T(K)\circ t\circ K$.
\end{defn}

\indent
With the same notations as in Definition \ref{defn:qFueter}\,, let $\Cal_T$ be the
endomorphism defined by
$\Cal_T(t)=T(I)\circ t\circ I+T(J)\circ t\circ J+T(K)\circ t\circ K$, $(t\in{\rm Hom}_{\R}(V,W)\,)$\,,
where $(I,J,K)$ is a positive orthonormal basis of $Q_V$\,. A straightforward calculation
shows that $\Cal_T$ does not depend of the positive orthonormal basis $(I,J,K)$ and,
in particular, the notion of `Fueter-quaternionic map', between almost quaternionic manifolds,
is well-defined.\\
\indent
Furthermore, $\Cal_T$ satisfies the equation $(\Cal_T)^2+2\,\Cal_T-3=0$\,.
Let $\F_T$ and $\Q_T$ be the eigenspaces of $\Cal_T$ corresponding to $1$ and $-3$\,, respectively.
Then ${\rm Hom}_{\R}(V,W)=\Q_T\oplus\F_T$ and $\F_T$ is the space of
linear Fueter-quaternionic maps, with respect to $T$, whilst $\Q_T$ is the space of
linear quaternionic maps, with respect to $T$ \cite{Hay}\,.
Apparently, this would suggest that Fueter-quaternionic maps are `anti-quaternionic'.
In fact, by reformulating results mentioned in \cite{CheLi} and \cite{Hay}\,, the following proposition
can be easily obtained.

\begin{prop} \label{prop:QF}
Let $V$\! and $W$ be quaternionic vector spaces and let $T$ be an orientation preserving isometry
from $Z_V$ to $Z_W$\,.\\
\indent
Then, for any line through the origin $d\subseteq Q_V$\,, we have $\Q_{T\circ S_d}\subseteq\F_T$\,,
where $S_d$ is the symmetry in $d$\,. Moreover, $\F_T$ is generated by $\bigcup_d\Q_{T\circ S_d}$\,.
\end{prop}

\indent
Finally, let $U$, $V$ and $W$ be quaternionic vector spaces and let $T':Z_U\to Z_V$ and
$T'':Z_V\to Z_W$ be orientation preserving isometries. Also, let $d\subseteq Q_V$ be a line through the origin.
If $t':U\to V$ and $t'':V\to W$ are linear quaternionic maps, with respect to $S_d\circ T'$ and $T''\circ S_d$\,,
respectively, then, by Proposition \ref{prop:QF}\,, $t'$ and $t''$ are Fueter-quaternionic,
with respect to $T'$ and $T''$, respectively. However, $t''\circ t'$ is Fueter-quaternionic,
with respect to $T''\circ T'$, if and only if $t''\circ t'=0$\,.

\end{document}